\documentclass[12pt,reqno,a4paper]{amsart}
\usepackage{
    amsmath,  amsfonts, amssymb,  amsthm,   amscd,
    gensymb,  graphicx, comment,  etoolbox, url,
    booktabs, stackrel, mathtools,enumitem, mathdots,  microtype, lmodern,    mathrsfs, graphicx, tikz,  longtable,tabularx, float, tikz, pst-node, tikz-cd, multirow, tabularx, amscd,  bm, array, makecell, diagbox, booktabs,ragged2e, caption, subcaption }
\usepackage{makecell,slashbox}

\usepackage{xcolor}
\usepackage[utf8]{inputenc}
\usepackage{microtype, fullpage, wrapfig,textcomp,mathrsfs,csquotes,fbb}
\usepackage[colorlinks=true, linkcolor=blue, citecolor=blue, urlcolor=blue, breaklinks=true]{hyperref}
\usepackage[capitalise]{cleveref}
\usepackage{todonotes}
\usetikzlibrary{positioning}
\usetikzlibrary{shapes,arrows.meta,calc}
\usetikzlibrary{arrows}

\newtheorem{theorem}{Theorem}[section]

\newtheorem{definition}[theorem]{Definition}

\newtheorem{proposition}[theorem]{Proposition}
\newtheorem{remark}[theorem]{Remark}

\newcommand\Q{\mathbb{Q}}
\newcommand\R{\mathbb{R}}
\newcommand\Z{\mathbb{Z}}
\newcommand\F{\mathbb{F}}
\newcommand\C{\mathbb{C} }

\newcommand{\TC}{\mathrm{TC}}

\newcommand{\ct}{\mathrm{cat}}
\newcommand{\cl}{\mathrm{cup}}

\newcommand{\sct}{\mathrm{secat}}

\newcolumntype{x}[1]{>{\centering\arraybackslash}p{#1}}

\begin{document}
\begin{center}
{ 
       {\Large \textbf { \sc Higher (equivariant) topological complexity of Milnor manifolds
          }
       }
\\

\medskip

 {\sc Navnath Daundkar}\\
{\footnotesize Indian Institute of Science Education and Research Pune, India}\\

{\footnotesize e-mail: {\it navnath.daundkar@acads.iiserpune.ac.in}}}\\

 {\sc Bittu Singh}\\
{\footnotesize Indian Institute of Technology Bombay, India}\\

{\footnotesize e-mail: {\it 22d0781@iitb.ac.in}}\\

\end{center}

\medskip

\begin{center}
  {\sc Abstract}\\
\end{center}
J. Milnor introduced a specific class of codimension-$1$ submanifolds in the product of projective spaces, known as Milnor manifolds. This paper establishes precise bounds on the higher topological complexity of these manifolds and provides exact values for this invariant for numerous Milnor manifolds.   
Furthermore, we improve the upper bounds on the higher equivariant topological complexity. As an application, we obtain sharper bounds on the higher equivariant topological complexity of Milnor manifolds with free $\mathbb{Z}_2$ and $S^1$-actions.

\hrulefill



\section{Introduction and background}   \label{sec:intro}
A motion planning algorithm for a mechanical system is a function that associates a pair of its states $(x,y)$, with a continuous motion from $x$ to $y$. 
In other words, a \emph{motion planning algorithm} can be understood as a section of the free path space fibration $\pi:X^I\to X\times X,$ defined by
\[\pi(\gamma)=(\gamma(0),\gamma(1)), \]
where $X^I$ is the free path space of $X$, equipped with the compact open topology.
Farber \cite{FarberTC} introduced 
the notion of topological complexity
 to compute the complexity of the problem of finding a motion planning algorithm for the configuration space $X$ of a mechanical system.  
The \emph{topological complexity} of a space $X$, denoted by $\TC(X)$ is defined as the least positive integer $r$ for which $X\times X$ can be covered by open sets $\{U_1,\dots, U_r\}$,  where each $U_i$ admits a continuous local section of $\pi$. 
Farber \cite[Theorem 3]{FarberTC} showed that $\TC(X)$ is a numerical homotopy invariant of a space $X$. 

The higher analogue of topological complexity was introduced by Rudyak in \cite{RUD2010}. 
For a path-connected space $X$, consider the fibration $\pi_n: X^I\to X^n$ defined by
\begin{equation}\label{eq: pik fibration}
  \pi_n(\gamma)=\bigg(\gamma(0), \gamma\bigg(\frac{1}{n-1}\bigg),\dots,\gamma\bigg(\frac{n-2}{n-1}\bigg),\gamma(1)\bigg).   
\end{equation}
The \emph{higher topological complexity} of $X$,  denoted by $\TC_n(X)$, is the least positive integer $r$ for which $X^n$ can be covered by open sets $\{U_1,\dots, U_r\}$,  where each $U_i$ admits a continuous local section of $\pi_n$ .
Note that when $n=2$, $\TC_n(X)$ coincides with $\TC(X)$.

There is an old invariant of topological spaces called the \emph{LS category},  introduced by Lusternik and Schnirelmann in \cite{LScat}. 
The LS category of a space $X$, denoted by $\mathrm{cat}(X)$, is the smallest positive integer $r$ such that $X$ can be covered by $r$ open subsets $V_1, \dots, V_r$, where each inclusion $V_i\xhookrightarrow{} X$ is null-homotopic.
The following inequalities were established in  \cite{gonzalezhighertc}: 
\begin{equation}\label{eq: tcn lbub}
\mathrm{cat}(X^{n-1})\leq \TC_n(X)\leq \mathrm{cat}(X^n).
\end{equation}

These invariants are special cases of a broader concept known as the sectional category of a map. More generally, the equivariant sectional category of an equivariant map was introduced in \cite{colmangranteqtc}. 
Let  $p\colon E \to B$ be a $G$-map. 
The \emph{equivariant sectional category} of $p$, denoted by $\sct_G(p)$, is the least positive integer $r$ such that there exists a $G$-invariant open cover $\{V_1, \dots, V_r\}$ of $B$, and for each $1\leq i\leq r$, a $G$-map $\sigma_i \colon V_i \to E$ satisfying $f \sigma_i \simeq_G \iota_{V_i} \colon V_i \hookrightarrow B$. 
If no such $r$ exists, we say $\sct_G(p)=\infty$. 
We note that for $G$-fibrations, we can replace $\simeq_G$ by $=$ in the above definition. 
Observe that the path space  $B^I$ admits a $G$-action defined by $(g \cdot \gamma) (t) = g\gamma(t)$.
The fibration $\pi_n \colon B^I \to B^n$ as defined in \eqref{eq: pik fibration}, is also a $G$-fibration. 
The \emph{ higher equivariant topological complexity} of $B$, denoted by $\TC_{G,n}(B)$, is defined as  $\TC_{G,n}(B) := \sct_G(\pi_n)$, as discussed in \cite{byehsarkareqtcn}.

Determining the precise values of these invariants is often a challenging task. Over the past two decades, several mathematicians have significantly contributed to approximate these invariants with bounds. 
To be more specific, Farber \cite[Theorem 7]{FarberTC} gave a cohomological lower bound on the topological complexity, and this concept was extended to the higher topological complexity by Rudyak in \cite[Proposition 3.4]{RUD2010}. 
Let $d_n:X\to X^n$ be the diagonal map, and let  $\cl(X,n)$ denote the cup-length
of elements in the kernel of the map induced in cohomology by $d_n$. 
More precisely, $\cl(X, n)$ is the
largest integer $m$ for which there exist cohomology classes $u_i \in H^{\ast}(X^n; G_i)$ such that $d_n^{\ast}(u_i)=0$, and $\prod_{i=1}^{m}u_i\neq 0$ in $H^{\ast}(X^n;\otimes_{i=1}^m G_i)$, where $G_i$ are abelian groups for $1\leq i\leq m$. Then, \cite[Proposition 3.4]{RUD2010}  shows
\begin{equation}\label{eq: lb higher tc}
 \cl(X,n)+1\leq \TC_n(X).   
\end{equation}
For $n=2$, the bound in \eqref{eq: lb higher tc} was established by Farber and it  
is known as the \emph{zero-divisors-cup length}. 
We refer to the non-negative integer $\cl(X,n)$ as the \emph{higher zero-divisors-cup-length} of $X$. 
Let $\cl(X)$ be the cup length of a cohomology ring of $X$. Then the following inequality was shown in \cite[Proposition 1.5]{CLOT}.
\begin{equation}\label{eq: cup-length}
\cl(X)+1 \leq \ct(X).    
\end{equation}  
For a fibration $F\hookrightarrow E \stackrel{p}{\rightarrow} B$ with a paracompact space $B$, the dimension-connectivity upper bound on the $\sct(p)$ established in \cite[Theorem 5]{Svarc61} is  given as follows:
\begin{equation}
    \sct(p) \leq \frac{\mathrm{dim}(B)}{\mathrm{conn}(F)+2}+1,
\end{equation}
where, $\mathrm{conn}(F)$ is the connectivity of $F$ and $\mathrm{dim}(B)$ is the dimension of $B$.
In particular, if $X$ is $m$-connected, then we have
\begin{equation}\label{eq: usual dim ub}
 \TC_n(X)\leq \frac{n\cdot\mathrm{hdim}(X)}{m+1}+1 , 
\end{equation}
\begin{equation}\label{thm: catdimub}
\ct(X)\leq \frac{\mathrm{hdim}(X)}{m+1}+1.
\end{equation}
Here, dimension is replaced by homotopy dimension because these invariants are homotopy invariants.

In this paper, we investigate the higher (equivariant) topological complexity of Milnor manifolds.
We begin with a brief overview of Milnor manifolds.
In his work \cite{Milnor}, Milnor introduced a class of submanifolds of the products of real and complex projective spaces to define generators for the unoriented cobordism algebra. 
Let $r$ and $s$ be integers such that $0\leq s\leq r$.
A \emph{Milnor manifold}, denoted by $\mathbb{F}H_{r,s}$, is defined as follows:
\[\mathbb{F}H_{r,s}:=\Big\{\big([z_0:\dots:z_r],[w_0:\dots:w_s]\big)\in \F P^r \times \F P^s\mid z_0\overline{w}_0+\cdots+z_s\overline{w}_s=0\Big\}.\] It is called a
\emph{real or complex Milnor manifold} if $\F=\R $ or $\F=\C,$ respectively.
The manifold $\mathbb{F}H_{r,s}$ is a closed, smooth manifold of dimension $k(s+r-1)$, where $k=1$ if $\F =\R$ and  $k=2$ if $\F=\C$.
There is a fiber bundle:
\begin{equation}\label{eq: fbRHrs}
 \mathbb{F}P^{r-1} \stackrel{i}{\hookrightarrow} \mathbb{F}H_{r,s} \stackrel{p}{\longrightarrow} \mathbb{F}P^{s}.
\end{equation}
This fiber bundle can be thought of as a projectivization of a vector bundle whose fiber over a point $\ell\in \F P^s$ is the $r$-dimensional plane $\ell^{\perp}$ in $\F^{r+1}$.
Milnor \cite[Lemma 1]{Milnor} established that the unoriented cobordism algebra of smooth manifolds is
generated by the cobordism classes of real projective spaces and real Milnor manifolds. 
Dey and Singh, in their work \cite{DeyMilnormfds}, provided a characterization of Milnor manifolds that admit free $\Z_2$ and $S^1$-actions, with constraints on $r$, and computed the corresponding equivariant cohomology rings.

Since a fiber bundle generalizes a trivial bundle, it is essential to examine how the higher topological complexity of a fiber bundle's total space compares to that of the trivial bundles total space, (i.e., the product of fiber and base spaces).
Since $\F H_{r,s}$ is a nontrivial $\F P^{r-1}$-bundle over $\F P^s$, it is meaningful to compare the values $\TC_n(\F H_{r,s})$ and $\TC_n(\F P^{r-1}\times \F P^{s})$.
In another notable result,
Farber, Tabichnikov and Yuzvinsky \cite{FarberTCproj} showed a compelling connection between the topological complexity of a real projective space and its immersion dimension. 
They proved that, 
\[\TC(\R P^s)=\mathrm{Imm}(\R P^s)+1\]
except $s=1,3,7$, where $\mathrm{Imm}(\R P^s)$ stands for the smallest dimension of the Euclidean space where $\R P^s$ can
be immersed.
It is worth noting that the Euclidean immersion problem of real projective spaces remains unsolved.
In \cite{tcnproj}, authors showed that the higher topological complexity of a real projective space $\TC_n(\R P^s)$ is related to its immersion dimension $\mathrm{Imm}(\R P^s)$. 
These results show that the information of geometric invariants of manifolds can be derived by investigating motion planning problems associated with them.
Furthermore, (real) Milnor manifolds are an important class of manifolds closely related to (real) projective spaces. Hence, it is important to examine the higher topological complexity of Milnor manifolds and explore the potential deduction of geometric properties associated with them.

The LS category and sharp bounds on the equivariant topological complexity of Milnor manifolds were computed by the first author in \cite{TCmilnor}.
This paper aims to study the higher equivariant topological complexity of Milnor manifolds. 
In \Cref{sec:tcn}, we compute the sharp bounds on the higher zero-divisors-cup-length of Milnor manifolds, showing that, in several cases, the $n^{th}$ topological complexity of real Milnor manifolds can take one of three values: $nd-1, nd, nd+1$, where $d$ represents their dimension (see \Cref{thm: tcn fhrs}). 
We determine the exact values of the higher topological complexity for specific Milnor manifolds $\R H_{r,1}$, where $r$ is a power of $2$ (see \Cref{thm: tcn r2t}).
Furthermore, we demonstrate that, in many cases, $\TC_n(\R H_{r,s})$ can differs from $\TC_n(\R P^{r-1}\times \R P^{s})$ by at most $2$, by applying Davis's work \cite{daviszclk}.
Finally, in \Cref{sec: eqtcn}, we compute bounds on the higher equivariant topological complexity of Milnor manifolds.

\vspace{0.3cm}

\noindent \textbf{Acknowledgement:}
The authors sincerely thank the reviewer for numerous insightful suggestions, comments, and feedback, which significantly improved the clarity and presentation of this article. Special thanks go to Prof. Rekha Santhanam for her constant encouragement and many useful discussions related to the topological complexity.
Navnath Daundkar thanks NBHM for the support through grant 0204/10/(16)/2023/R\&D-II/2789. The second author is grateful for support from the Prime Minister's Research Fellowship (PMRF ID 1302644), Government of India.

\section{Higher topological complexity of Milnor manifolds}\label{sec:tcn}
In this section, we review relevant results on Milnor manifolds,
which we will use to compute bounds on their higher topological complexity.

We begin by recalling the descriptions of the cohomology rings of $\R H_{r,s}$ and $\C H_{r,s}$, as obtained in \cite[Theorem 5.9]{Buchstaber} and \cite[Section 3.2]{Mukherjee}.
\begin{theorem}[{}]\label{thm: coring}
The mod-$2$ cohomology ring of $\R H_{r,s}$ is given by the following expression:
\[H^{\ast}(\R H_{r,s}; \Z_2) \cong \displaystyle\frac{\Z_2[a,b]}{\langle a^{s+1}, b^r+ab^{r-1}+ \dots+a^sb^{r-s} \rangle},\]
where both $a$ and $b$ are homogeneous elements of degree one.
\end{theorem}
It is known from \cite[Page 161]{Mukherjee} that a basis of $H^{\ast}(\R H_{r,s};\Z_2)$ is given by 
\begin{equation}\label{eq: basis}
\{a^ib^j \mid 0\leq i\leq s~~, ~~ 0\leq j\leq r-1\}  .  
\end{equation}
We also note the following relations arise from the cohomology ring given in \cref{thm: coring}, which will be useful in subsequent cohomological calculations:
\begin{equation}\label{eq: rel}
 \begin{split}
  a^{s-1}b^r&=a^sb^{r-1},\\
     a^{s-2}b^r&=a^{s-1}b^{r-1}+a^{s}b^{r-2},\\
    a^{s-3}b^r&=a^{s-2}b^{r-1}+a^{s-1}b^{r-2}+ a^sb^{r-3}.     
 \end{split}  
\end{equation}

The description of the integral cohomology ring of $\C H_{r,s}$ was provided in \cite[Lemma 3]{Buchstaber}. 
We will now state the corresponding rational version.

\begin{theorem}\label{thm: coring CHrs}
The rational cohomology ring of $\C H_{r,s}$ is given by the following expression:
\[H^{\ast}(\C H_{r,s}; \Q) \cong \frac{\Q[x,y]}{\langle x^{s+1},y^r+xy^{r-1}+\dots+x^sy^{r-s} \rangle},\]
where both $x$ and $y$ are homogeneous elements of degree two.   
\end{theorem}

 In their work \cite{DeyMilnormfds}, Dey and Singh characterized Milnor manifolds which admit free $\Z_2$ and $S^1$-actions under certain conditions. 
We list some of their relevant results, that are useful in computing lower bounds on the higher (equivariant) topological complexity of  Milnor manifolds.
\begin{theorem}[{\cite[Corollary 5.6, Corollary 5.8]{DeyMilnormfds}}\label{thm: existence of invo}\label{thm: free invo milnor mfds}]
Let $1<s<r$ such that $r\not\equiv2({\rm mod}~4)$.
Then $\F H_{r,s}$ admits a fixed-point-free involution if and only if both $r$ and $s$ are odd integers.
\end{theorem}

\begin{theorem}[{\cite[Proposition 5.9]{DeyMilnormfds}}\label{thm: free S1 action}]
Let $1\leq s\leq r$. Then $S^1$ acts freely on $\R H_{r,s}$ if and only if both $r$ and $s$ are odd integers.   
\end{theorem}

Throughout the paper, we fix the following notations: For $1\leq i\leq n$, the map $p_i: X^n \rightarrow X$ denotes the projection onto the $i$-th factor of $X$. For a cohomology class $x\in H^{\ast}(X; R)$, the class $x_i$ denotes the pullback $p_i^{\ast}(x)$ for $1\leq i\leq n$. The class $\bar{x}_i=x_i-x_{i-1}$ denotes the higher zero divisors for $2\leq i\leq n$. 

It was shown in \cite{TCmilnor} that, the LS category of $\F H_{r,s}$ is $s+r$.
We compute $\ct(\F H_{r,s})^n$ for any $n$ and  derive bounds on $\TC_n(\R H_{r,s})$.
\begin{proposition}\label{prop:cat-FHrsn}
Let $n$ be a natural number. Then 
\begin{equation}\label{eq: ctfhrs}
  \ct((\F H_{r,s})^n) =n(r+s -1) +1.
\end{equation}
 For $n\geq 2$, we have
\begin{equation}\label{eq: tcn LBUB}
   (n-1)(r+s-1)+1\leq \TC_n(\R H_{r,s}) \leq n(r+s-1)+1.
\end{equation}
Moreover, if both $r$ and $s$ are odd, then $\TC_n(\R H_{r,s}) \leq n(r+s-1)$.
\end{proposition}
\begin{proof}
Let $X=\R H_{r,s}$.
Observe that the following cohomology class \[\prod_{i=1}^n a_{i}^{s} b_{i}^{r-1} \] is a generator of $H^{n(r+s-1)}(X)$.  
Thus, from \eqref{eq: cup-length}, we get $\ct(X^n)\geq n(s+r-1)+1$. 
The inequality $\ct(X^n)\leq n(s+r-1)+1$  follows from \eqref{thm: catdimub}.
Since the cohomology ring of $\C H_{r,s}$ has a description similar to that of $\R H_{r,s}$, one can derive the
inequality, $\ct(\C H_{r,s})^n\geq n(s+r-1)+1$.
Since $\C H_{r,s}$ is simply connected, the inequality $\ct((\C H_{r,s})^n)\leq n(s+r-1)+1$ follows from \eqref{thm: catdimub}.

From \eqref{eq: tcn lbub} and \eqref{eq: ctfhrs}, we conclude \eqref{eq: tcn LBUB}.
It follows from \Cref{thm: free S1 action} that,
when both $r$ and $s$ are odd, $\R H_{r,s}$ admits a free $S^1$-action. 
Consequently, the inequality $\TC_n(\R H_{r,s}) \leq n(r+s+1)$ follows from \cite[Corollary 4.7]{daundkar2023group}.
\end{proof}

We will now compute the exact value of $\TC_n(\C H_{r,s})$.
\begin{theorem}
The higher topological complexity of $\C H_{r,s}$ is given by
 $$\TC_n(\C H_{r,s}) =n(r+s -1) +1.$$
\end{theorem}
\begin{proof}
For $x\in H^{\ast}(\C H_{r,s};\Q)$, using \Cref{thm: coring CHrs}, we have
 \begin{align*}
        \bar{x}_i^{2s-1}&=(x_i-x_{i-1})^{2s-1}\\
        &= \sum_{j=s-1}^s (-1)^{2s-1-j} \binom{2s-1}{j} x_i^j x_{i-1}^{2s-1-j}\\
        &=\gamma (x_i^{s-1}x_{i-1}^s- x_i^{s}x_{i-1}^{s-1}),
\end{align*}
where $\gamma=(-1)^{s}\binom{2s-1}{s-1}.$
For $y\in H^{\ast}(\C H_{r,s};\Q)$, we obtain the similar expression $$\bar y_i^{2r-1}=\gamma' (y_i^{r-1}y_{i-1}^r- y_i^{r}y_{i-1}^{r-1}),$$
where $\gamma'=(-1)^{r}\binom{2r-1}{r-1}.$ 
Since $x^sy^r=0$ follows from dimensional reason, we now have
\begin{align*}
    \bar{x}_i^{2s-1}\bar{y}_i^{2r-1}&=-\gamma \gamma' (x_i^{s-1}x_{i-1}^sy_i^ry_{i-1}^{r-1}+x_i^{s}x_{i-1}^{s-1}y_i^{r-1}y_{i-1}^{r})\\
    &= 2\gamma \gamma'(x_{i-1}^sy_{i-1}^{r-1})(x_i^s y_i^{r-1}).
\end{align*}
We have used the fact that $x^sy^{r-1}+x^{s-1}y^r=0$, which follows from \cref{thm: coring CHrs}, in the second equality of the previous expression.

\noindent \textbf{Case-1:}\emph{ $n=2k$}.\\
Consider the following product \begin{align*}
    \prod_{i=1}^k  \bar{x}_{2i}^{2s-1}\bar{y}_{2i}^{2r-1}=(2\gamma \gamma')^k \prod_{i=1}^{2k} x_i^sy_i^{r-1}.
\end{align*}
Since $x_i^sy_i^{r-1}$ is the pullback of the generator of $H^\ast (\C H_{r,s})$, the above product is non-zero.
Therefore, \eqref{eq: lb higher tc} gives, $$\TC_n(\C H_{r,s})
\geq k\{(2r-1)+(2s-1)\}+1=n(r+s-1)+1.$$
\noindent \textbf{Case-2:}\emph{ $n=2k+1$}.\\
Consider the product
\[(\prod_{i=1}^k  \bar{x}_{2i}^{2s-1}\bar{y}_{2i}^{2r-1})(\bar x_{2k+1}^s \bar y_{2k+1}^{r-1})=(2\gamma \gamma')^k \prod_{i=1}^{2k+1} x_i^sy_i^{r-1} \neq 0.\]
Thus, for any $n$, we have $$\TC_n(\C H_{r,s})
\geq n(r+s-1)+1.$$
Since $\C H_{r,s}$ is simply connected, the other inequality then follows from \eqref{eq: usual dim ub}.
\end{proof}

In the following theorem, we establish a lower bound for $\TC_n(\R H_{r,s})$ that is stronger than the one in \Cref{prop:cat-FHrsn} when $r$ is a power of $2$ and compute the exact values in several specific cases.

\begin{theorem}\label{thm: tcn r2t}
Let $r$ be a power of $2$ ($r=2^t$, $t\geq 0$) and $n\geq 3$. 
 Then   
 \begin{equation}\label{eq: tcn r2t}
  n(s+r-1)-s+2  \leq  \TC_n(\R H_{r,s})\leq n(s+r-1)+1.
 \end{equation}
 In particular, if $s=1$, then $\TC_n(\R H_{r,1})=nr+1$.
\end{theorem}
\begin{proof}

Let $b\in H^{\ast}(\R H_{r,s}; \Z_2)$.
Since $\binom{2^t}{j}=1$(mod 2) if and only if $j=0$ or $2^t$ (a consequence of Lucas's theorem), we obtain
$\bar{b}_{i}^{r}=b_{i-1}^r+b_i^r$. 
Taking the product of $\bar{b}_{i}^{r}$'s gives us the following expression \[\prod_{i=2}^n\bar{b}_i^r=\sum_{i=1}^n\prod_{\substack{j=1\\ j \neq i}}^nb_j^r.\]
Here, we have used the fact that $b^{r+j}=0$ for $j \geq 1$ which follows from \Cref{thm: coring}.
Next multiplying $\prod_{i=2}^n\bar{b}_i^r$ by $\bar{b}_2^{r-1}$, yields
\begin{align*}
    \bar{b}_2^{r-1}\prod_{i=2}^n\bar{b}_i^r &=\bigg( \sum_{j=1}^n b_1^j b_2^{r-1-j} \bigg)  \bigg(\sum_{i=1}^n\prod_{\substack{j=1\\ j \neq i}}^nb_j^r \bigg)\\
    &=b_1^{r-1}\prod_{\substack{j=1\\j \neq 1}}^nb_j^r + b_2^{r-1}  \prod_{\substack{j=1\\ j \neq 2}}^nb_j^r.
\end{align*}
Now, multiplying by $\bar b_3$ to $\bar{b}_2^{r-1}\prod_{i=2}^n\bar{b}_i^r$ (since $n\geq 3$), we obtain:
\[\bar{b}_2^{2r-1}\bar{b}_{3}^{r+1}\prod_{i=4}^n\bar{b}_i^r=\prod_{j=1}^nb_j^r.\]

Note that the product $\prod_{i=2}^n(a_1+a_i)^{s-1}$ contains the unique term $\prod_{i=2}^na_i^{s-1}$.
Therefore. the following product
\[\bigg( \bar b_2^{2r-1} \bar{b}_{3}^{r+1}\prod_{i=4}^n\bar{b}_i^r \bigg)  \prod_{i=2}^n(a_1+a_i)^{s-1}\] contains the term $b_1^r\prod_{i=2}^na_i^{s-1}b_i^r$ which cannot be killed by any other terms in the product.
Consequently, using \eqref{eq: lb higher tc}, we obtain the following inequality  \[\TC_n(\R H_{r,s})\geq n(s+r-1)-s+2.\]
The other inequality of \eqref{eq: tcn r2t} follows from \eqref{eq: tcn LBUB}.
\end{proof}

\begin{remark}
Note that if $s=1$ and $r=2$, then $\R H_{2,1}$ is homeomorphic to the Klein bottle $K$.  
Furthermore, \Cref{thm: tcn r2t} shows that $\TC_n(K)=2n+1$ for $n\geq 3$. 
Therefore, the proof of  \Cref{thm: tcn r2t} recovers the result \cite[Proposition 5.1]{Flag} for the Klein bottle case. Note that \cite[Proposition 5.1]{Flag} computes the higher topological complexity of closed surfaces, except for the sphere and the torus.
\end{remark}

 In \cite[Theorem 3.2]{TCmilnor}, the first author of this paper showed that if $s=2^{t_1}+1$ and $r=2^{t_2}$, then $\TC(\R H_{r,s})$ is either $2(s+r-1)-1$ or $2(s+r-1)$. 
In this case, the cohomological methods were used to compute the lower bound on $\TC(\R H_{r,s})$. 
The non-maximality of $\TC(\R H_{r,s})$ followed using the fact that the $\pi_1(\R H_{r,s})=\Z_2\times \Z_2$ when $r>2$ and $s>1$, and using Cohen and Vandembroucq's results from \cite{abfundamentaltc}.

In the following result, we estimate the number $\mathrm{cup}(\R H_{r,s},n)$ to obtain the sharp bounds on $\TC_n(\R H_{r,s})$ in several cases.
\begin{theorem}\label{thm: tcn fhrs}
Suppose one of the following holds:
\begin{enumerate}
    \item $s= 2^{t_1} +1$ and $r=2^{t_2}$ for some positive integers $t_1, t_2$.
    \item  $s= 2^{p_1}$ and $r=2^{p_2}+1$ for some positive integers $p_1, p_2$.
\end{enumerate}
Then, for $n\geq 2$, we have
\begin{equation}\label{eq: tc ineq}
n(s+r-1)-1 \leq \TC_n(\R H_{r,s})\leq n(s+r-1)+1.
\end{equation}  
\end{theorem}
\begin{proof}
\noindent \textbf{Case-1:} \emph{$s= 2^{t_1} +1$ and $r=2^{t_2}$.}\\ 
We now use the cohomology ring description of $\R H_{r,s}$, as provided in \Cref{thm: coring}, to derive the following equality. Consider the following expression
\begin{align*}
    \bar{a}_i ^ {2(s-1)-1} &= \sum _{j=0}^{2s-3} \binom {2s-3} {j} a_i^{2s-j-3}  a_{i-1}^{j}\\
    &= \sum _{j=0}^{3} \binom {2s-3} {s-j} a_{i}^{s+j-3}  a_{i-1}^{s-j}  ~~(\text{follows from \Cref{thm: coring}}).
\end{align*}
Using Lucas's theorem we have $\binom{2^{i-1}}{j}=1$ (mod $2$), which allows the expression to reduce
 \[ \bar{a}_i ^ {2(s-1)-1}=\sum _{j=0}^{3} a_{i}^{s+j-3}  a_{i-1}^{s-j}, ~~\text{for}~~ 2\leq i\leq n. \]
Observe that $\bar{a}_i^{2(s-1)-1}$ is nonzero, as the elements of the form  $a_{i}^{s+j-3}  a_{i-1}^{s-j} $ are part of a basis of $H^{2s-3}((\R H_{r,s})^n;\Z_2)$.
We also have the following similar expression for $2\leq i\leq n$.
\begin{align*}
    \bar{b}_i ^ {2r-1} &= \sum _{j=0}^{2r-1} \binom {2r-1} {j} b_{i}^{2r-1-j}  b_{i-1}^{j}\\
    &= \sum _{j=0}^{1} \binom {2r-1} {r-1+j} b_{i}^{r-j}  b_{i-1}^{r-1+j}  ~~(\text{follows from \Cref{thm: coring}}).\\
    &= \sum _{j=0}^{1} b_{i}^{r-j}  b_{i-1}^{r-1+j}.
\end{align*}
Observe that $\bar{b}_i ^ {2r-1}$ is nonzero
as $b_{i}^{r-j}  b_{i-1}^{r-1+j}$  is part of a basis of $H^{2r-1}((\R H_{r,s})^n)$.

We will now examine the case when $n=2k$. Consider the following nonzero products
\begin{align*}
    \prod_{i=1}^{k}\bar{a}_{2i} ^ {2(s-1)-1} 
    &= \prod_{i=1}^{k} \sum _{j=0}^{3} a_{2i}^{s+j-3}  
    a_{2i-1}^{s-j}\\
    &= \sum _{\substack{ j_\nu=0\\1\leq \nu \leq k}}^{3} a_{1}^{s-j_1}  a_{2}^{s+j_{1}-3} \cdots a_{2k-1}^{s-j_k}  a_{2k}^{s+j_{k}-3},
\end{align*} 
and
\begin{align*}
    \prod_{i=1}^{k}\bar{b}_{2i} ^ {2r-1}
   &= \prod_{i=1}^{k} \sum _{j=0}^{1} b_{2i}^{r-j}  b_{2i-1}^{r-1+j}\\
    &= \sum_{\substack{ j_\nu=0\\1\leq \nu \leq k}}^{1} b_{1}^{r-1+j_1}  b_{2}^{r-j_1} \cdots b_{2k-1}^{r-1+j_k}  b_{2k}^{r-j_k}.
\end{align*}
Let $A= \prod_{i=1}^{k}\bar{a}_{2i} ^ {2(s-1)-1}$ and $B= \prod_{i=1}^{k}\bar{b}_{2i} ^ {2r-1}$.
Now, we express their product in a simplified form as follows:
\begin{align*}\label{eq: prod ab}
   A  B 
    &= \big( \sum _{\substack { j_\nu=0\\ 1 \leq \nu \leq k}}^{3} a_{1}^{s-j_1}  a_{2}^{s+j_{1}-3} \cdots a_{2k-1}^{s-j_k}  a_{2k}^{s+j_{k}-3} \big) 
      \big( \sum _{\substack{i_\mu=0\\ 1\leq \mu \leq k}}^{1} b_{1}^{r-1+i_1}  b_{2}^{r-i_1} \cdots b_{2k-1}^{r-1+i_k}  b_{2k}^{r-i_k} \big)\\  
     &= \sum _{\substack{i_\mu=0\\ 1\leq \mu \leq k}}^{1} \sum _{\substack { j_\nu=0\\ 1 \leq \nu \leq k}}^{3} (a_{1}^{s-j_1} b_{1}^{r-1+i_1})  {(a_{2}^{s+j_{1}-3} b_{2}^{r-i_1})} \cdots (a_{2k-1}^{s-j_k} b_{2k-1}^{r-1+i_k})  (a_{2k}^{s+j_{k}-3}   b_{2k}^{r-i_k}) \\
     &=\sum_{\substack {(j_\mu, i_\mu) \in S\\ 1\leq \mu \leq k}}  \prod_{p =1}^k (a_{2p  -1}^{s-j_p} b_{2p -1}^{r-1+i_p})  (a_{2 p }^{s+j_p-3} b_{2 p}^{r-i_p}),
  \end{align*}
  where \[S=\{(0,0),(1,0),(2,0),(1,1),(2,1),(3,1)\}.\]
Since the terms $a_{2p}^{s}b_{2p}^r$, $a_{2p-1}^{s}b_{2p-1}^r$ are zero, we do not considered the choices of $(j_{\mu},i_{\mu})=(3,0),(0,1)$ in the indexing set $S$.

Next, we write $A B$ in terms of basis elements described in \eqref{eq: basis} using the relations in \eqref{eq: rel}, which will suffice to say it is nonzero.
\begin{equation}\label{eq: AB}
A B =\sum_{\substack {(j_\mu, i_\mu,j_{\mu}', i_{\mu}' ) \in T\\ 1\leq \mu \leq k}}  \prod_{p =1}^k (a_{2p  -1}^{s-j_p} b_{2p -1}^{r-i_p})  (a_{2 p }^{s-j_p'} b_{2 p}^{r-i_p'}),    
\end{equation}
where $T=\{(0,1,1,2), (0,1,0,3), (1,1,0,2), (0,2,1,1), (1,2,0,1), (0,3,0,1)\}$.

Let $$C= \prod_{i=1}^{k-1} (a_{2i-1} +a_{2i+1}).$$
Observe that the exponent of $a_{2p-1}$ in \eqref{eq: AB} is either $s$ or $s-1$. Thus, it follows from \Cref{thm: coring} that the exponents of $a_{2p-1}$ can be increased by at most one so that it remains nonzero, yielding the following equality.
\begin{align*}
    & A  B  C = A  B  \big (\sum_{l=1}^k \prod _{\substack{j=1 \\ j \neq i}}^{k} a_{2j-1} \big)=\sum_{l=1}^k\Gamma_l,
 \end{align*}
where 
\begin{align*}
   \Gamma_l &= A B \bigg( \prod _{\substack{j=1 \\ j \neq i}}^{k} a_{2j-1} \bigg ) \\
&=   \sum_{\substack {( i_\mu, i_{\mu}' ) \in T'\\ 1\leq \mu \leq k}} \bigg\{ \prod_{p =1, p\neq l}^k (a_{2p  -1}^{s} b_{2p -1}^{r-i_p})  (a_{2 p }^{s} b_{2 p}^{r-i_p'}) \bigg ( \sum_{\substack {(j_l, i_l,j_{l}', i_{l}' ) \in T\\ 1\leq \mu \leq k}}  (a_{2l  -1}^{s-j_l} b_{2l -1}^{r-i_l})  (a_{2l }^{s-j_l'} b_{2l}^{r-i_l'})   \bigg )\bigg\},
\\  
 \end{align*}
where $T'=\{(1,2),(2,1)\}$. 
Let $$D= \prod_{i=1}^{k-1} (b_{2i-1} +b_{2i+1}).$$
Then,  using a similar reason as above, we obtain the following expression
\[A B C D=A  B  C \big (\sum_{q=1}^k \prod _{\substack{j=1 \\ j \neq q}}^{k} b_{2j-1} \big).\]
Next, we write
\[A B C D=\sum_{l=1}^k\sum_{q=1}^k\Gamma_{l,q},\]
where \[\Gamma_{l,q}=\Gamma_l \bigg ( \prod_{j=1, j\neq q}^kb_{2j-1} \bigg).\]
Now we observe that \[\Gamma_{l,l}=\prod_{p =1, p\neq l}^k (a_{2p  -1}^{s} b_{2p -1}^{r-1})  (a_{2 p }^{s} b_{2 p}^{r-1}) \bigg ( \sum_{\substack {(j_l, i_l,j_{l}', i_{l}' ) \in T\\ 1\leq \mu \leq k}}  (a_{2l  -1}^{s-j_l} b_{2l -1}^{r-i_l})  (a_{2l }^{s-j_l'} b_{2l}^{r-i_l'})   \bigg ).\]
When $l \neq q,$ we have
\begin{multline*}
         \Gamma_{l,q}=
\prod_{ \substack{p =1 \\ p\neq l,q } }^k (a_{2p  -1}^{s} b_{2p -1}^{r-1})  (a_{2 p }^{s} b_{2 p}^{r-1}) \bigg ( \sum_{ (i_q, i_q')\in T'} (a_{2q-1}^s b_{2q-1}^{r-i_q})  (a_{2q}^sb_{2q}^{r-i_q'}) \bigg ) \\ \cdot  \bigg ( \sum_{\substack {(j_l, i_l,j_{l}', i_{l}' ) \in T''\\ 1\leq \mu \leq k}}  (a_{2l  -1}^{s-j_l} b_{2l -1}^{r-i_l})  (a_{2l }^{s-j_l'} b_{2l}^{r-i_l'})   \bigg ),
\end{multline*}
where $T''=\{(0,1,0,2), (0,1,1,1),(1,1,0,1), (0,2,0,1)\}.$
Note that the above expressions are written in terms of the basis elements. Then
observing the positions of the highest degree terms $a^sb^{r-1}$, it is clear that $\Gamma_{l,l} \neq \Gamma_{l',l'} \neq \Gamma_{l,l'} $ whenever $l \neq l'$. Thus, we have $\sum_{l=1}^k\sum_{q=1}^k\Gamma_{l,q} \neq 0$, i.e., $ABCD \neq 0.$
Therefore, for $n=2k$, the cohomological lower bound in \eqref{eq: lb higher tc} yields
\begin{align*}
\TC_{n} (\mathbb RH_{r,s}) & \geq k(2(s-1) -1) +k(2r-1)+2(k-1) +1\\
&=2k (s+r-1)-1.\\
&=n (s+r-1)-1.
\end{align*}

To obtain a similar bound for $n=2k+1$, we consider the following product
\[A B C  D  ( \bar{a}_{2k+1}^{s} \bar{b}_{2k+1}^{r-1}). \] 
Observe that the above product is nonzero, as it contains the term $A  B  C  D (a_{2k+1}^s b_{2k+1}^{r-1})$,  which cannot be canceled out by any other term in the expression. 
Hence, for $n=2k+1$, the cohomological bound given in \eqref{eq: lb higher tc} provides
\begin{align*}
\TC_{n} (\mathbb RH_{r,s}) 
&\geq 2k (s+r-1)-1+(s+r-1)\\
&=n (s+r-1)-1.
\end{align*}
Thus, we conclude $\TC_n(\R H_{r,s})\geq n(s+r-1)-1$.
The upper bound of \eqref{eq: tc ineq} then follows from \Cref{eq: usual dim ub}.

\noindent \textbf{Case-2:} \emph{$s= 2^{p_1}$ and $r=2^{p_2}+1$}.\\  
This case can be approached in a manner similar to case-$1$, by reversing the roles of exponents of $\bar{a}_i$'s and $\bar{b}_j$'s. 
More precisely, when $n=2k$, we will take 
the product 
\[\prod_{p=1}^{k}\bar{a}_{2p} ^ {2s-1}  \prod_{q=1}^{k}\bar{b}_{2q} ^ {2(r-1)-1} C D,\] where $C$ and $D$ are defined in the first case.
Then one can note that the product 
\[\prod_{p=1}^{k}\bar{a}_{2p} ^ {2s-1}  \prod_{q=1}^{k}\bar{b}_{2q} ^ {2(r-1)-1}\] coincides with a similar expression given in \eqref{eq: AB}. 
Therefore, similar arguments as given in the first case apply here as well.
The case $n=2k+1$ can also be treated similarly.
This completes the proof.
\end{proof}
Now we will provide the bounds on $\TC_n(\R H _{r,s})$  for more cases. 


\begin{proposition}\label{prop: tcn  gen}
Let $r\geq 2^{t_1} +1$, $s \geq 2^{t_2}$ (or $s\geq 2^{t_1} +1$, $r \geq 2^{t_2}$ ) for some positive integers $t_1, t_2$. Then 
\begin{equation}\label{eq: tc ineq ineq}
n(2^{t_1}+ 2^{t_2})-1 \leq \TC_n(\R H_{r,s})\leq n(s+r-1)+1.
\end{equation} 
Moreover, if both $r$ and $s$ are odd, then 
\[n(2^{t_1}+ 2^{t_2})-1 \leq \TC_n(\R H_{r,s})\leq n(s+r-1).\]
\end{proposition}
\begin{proof}
Recall that the basis of $H^{\ast}(\R H_{r,s};\Z_2)$ is given by $\mathcal{B}=\{a^ib^{j} : 0\leq i\leq s ~~, 0\leq j\leq r-1\}$. Since $r\geq 2^{t_1} +1$, $s \geq 2^{t_2}$, $\mathcal{B}$ contains the basis of 
$H^{\ast}(\R H_{2^{t_1} +1,2^{t_2}};\Z_2)$ (or the basis of $H^{\ast}(\R H_{2^{t_2},2^{t_2}+1};\Z_2)$). 
This gives $$\cl(\R H_{2^{t_1} +1,2^{t_2}},n)\leq \cl(\R H_{r,s},n).$$
This observation, together with \eqref{eq: lb higher tc}, proves the left inequality of \eqref{eq: tc ineq ineq}. The case  $s\geq 2^{t_1} +1$, $r \geq 2^{t_2}$ follows similarly.
The inequality $\TC_n(\R H_{r,s})\leq n(s+r-1)$  then follows from \cite[Corollary 4.7]{daundkar2023group}.
\end{proof}

\begin{remark}
It was observed by the authors in \cite{Flag} that the higher analog of Farber-Costa's \cite{FC} results fail for real projective spaces.
At this point, it is unclear whether the higher analogs of the results of Cohen-Vandembroucq \cite{abfundamentaltc} hold true or not.
If they did, we could show that the higher topological complexity of real Milnor manifold is non-maximal when their fundamental group is $\Z_2\times \Z_2$.
Nevertheless, it would be interesting to see if the higher analog of the results of Cohen-Vandembroucq also fails in the case of real Milnor manifolds.
\end{remark}

Given that the proof of the subsequent proposition resembles that of \Cref{prop: tcn gen}, we choose not to repeat it in this instance.
\begin{proposition}
    Let $r \geq 2^t$. Then
     \begin{equation}
  n(2^t+s-1)-s+2  \leq  \TC_n(\R H_{r,s})\leq n (s+r-1)+1.
 \end{equation}
\end{proposition}

Recall from \cite{FarberTCproj} that if $m=2^t$, then $\TC(\R P^m)=2m$.
Davis \cite{daviszclk}, has computed the higher zero-divisors-cup-length of real projective spaces in several cases. Additionally,
the authors of \cite{tcnproj} have shown that under certain assumptions $\TC_n(\R P^m)$ is maximal if $m$ is even and $n>m$. 
In particular, the following result follows from \cite[Theorem 5.11, Proposition 6.2]{tcnproj}.

\begin{proposition}\label{prop: tcn rpm}
Let $m=2^t$ and $n\geq 3$. Then 
$\TC_n(\R P^m)= n2^t+1.    
$
\end{proposition}

Next, we observe that for $m=2^t+1$, the higher topological complexity of $\R P^m$ is one less than its usual dimensional upper bound.
\begin{proposition}\label{prop: tcnrpodd}
Let $m=2^t+1$. Then, for $n\geq 3$, we have $\TC_n(\R P^m)=nm$.    
\end{proposition}
\begin{proof}
Note that, if $m=2^t+1$, the equality $\cl(\R P^m,n)=nm-1$ follows from  \cite[Theorem 1.6]{daviszclk}.
Therefore, from \eqref{eq: lb higher tc}, we have $\TC_n(\R P^m)\geq nm$. 
Note that, the odd dimensional projective spaces admit a free circle-action.
 Therefore, $\TC_n(\R P^m)\leq nm-1+1$ follows from \cite[Corollary 4.7]{daundkar2023group}. This gives us the desired equality.
\end{proof}

The next result  establishes a connection between the  $\TC_n(\R H_{r,s})$, and $\TC_n(\R P^{r-1}\times \R P^s)$.
\begin{proposition}\label{thm: tcnrhrs rprprs}
Let $r= 2^{t_1 }+1$ and  $s=2^{t_2}$  for some positive integers $t_1, t_2$. Then, for $n\geq 3$, we have
\begin{equation}\label{eq: ngeq3-prod-proj-RHrs}
n (s+r-1) -1 \leq \TC_n(\mathbb RH_{r,s}), \TC_n(\mathbb RP^{r-1}\times \mathbb RP^s )-1 \leq n(s+r-1)+1 .    
\end{equation}
\end{proposition}
\begin{proof}
This follows from \Cref{prop: tcn rpm} and \Cref{thm: tcn fhrs}.
\end{proof}

\begin{remark}
It was previously shown in \cite{TCmilnor} that for $s=2^{t_1}+1$ and $r=2^{t_2}$, the difference between $\TC(\R H_{r,s})$ and $\TC(\R P^r\times \R P^s)$ can be at most be $1$.
\end{remark}

\section{Higher equivariant topological complexity of Milnor manifolds}\label{sec: eqtcn}
The concept of the equivariant topological complexity was first introduced by Colman and Grant in \cite{colmangranteqtc} and later extended by Bayeh and Sarkar in \cite{byehsarkareqtcn} as the higher equivariant topological complexity. In this section, we explore the higher equivariant topological complexity of Milnor manifolds.

In \cite{daundkar2023group}, the first author demonstrated that for certain $G$-spaces, the inequality \[\TC_n(X)\leq n\dim(X)-\dim(G)+1\] holds.
Here, we establish an analogous result for $\TC_{G,n}(X)$, when $X$ is a free, metrizable $G$-space.

\begin{proposition}\label{prop: eqtcn improved ub}
Let $X$ be a free metrizable $G$-space. Then
\[\TC_{G,n}(X)\leq n\cdot\mathrm{dim}(X)- \mathrm{dim}(G)+1.\]
\end{proposition}
\begin{proof}
It follows from \cite[Proposition 3.17]{byehsarkareqtcn} that \[\TC_{G,n}(X)\leq \ct_G(X^n).\]
Since $X^n$ is a free metrizable space, \cite[Proposition 1.15]{Eqlscategory} gives $\ct_G(X^n)=\ct(X^n/G)$.
Therefore, using \eqref{thm: catdimub} we obtain 
\[\TC_{G,n}(X)\leq \ct_G(X^n)=\ct(X^n/G)\leq \mathrm{dim}(X^n/G)+1= n\cdot \
\mathrm{dim}(X)-\dim(G)+1. \qedhere\]
\end{proof}

We now compute sharp bounds on $\TC_{\Z_2,n}(\R H_{r,s})$ for certain values of $r$ and $s$.
\begin{proposition}\label{prop: z2eqtcn}
Let $r,s$ be odd integers with $r >s$ and $r\geq 2^{t_1} +1$, $s \geq 2^{t_2}$ (or $s\geq 2^{t_1} +1$, $r \geq 2^{t_2}$). Then
\begin{equation}\label{eq: tcz2 ineq}
n(2^{t_1}+ 2^{t_2})-1 \leq \TC_{\Z_2, n}(\R H_{r,s})\leq  n(s+r-1)+1.
\end{equation}
\end{proposition}
\begin{proof}
Recall from \cite[Corollary 3.8]{byehsarkareqtcn} that, the inequality
$\TC_n(X)\leq \TC_{G,n}(X)$ holds for any $G$-space $X$.
Therefore, \Cref{prop: tcn  gen} gives us the left inequality of \eqref{eq: tcz2 ineq}.
 Now observe that, under the hypothesis, \Cref{thm: free invo milnor mfds} imply that $\R H_{r,s}$ admits a free $\Z_2$-action.
Therefore, using  \Cref{prop: eqtcn improved ub} we get 
\[ \TC_{\Z_2,n}(\R H_{r,s})\leq n(s+r-1)+1.\qedhere\]
 \end{proof}

\begin{proposition}
Let $r,s $ be odd integers with $r >s$ and $r \geq 2^t.$
Then, for $n\geq 3$, we have
\begin{equation}
  n(2^t+s-1)-s+2  \leq  \TC_{\Z_2, n}(\R H_{r,s})\leq n(s+r-1)+1. 
 \end{equation}
\end{proposition}
\begin{proof}
The proof is similar to that of \Cref{prop: z2eqtcn}. \end{proof}

\begin{proposition}\label{prop: s1eqtcn}
Let $r$ and $s$ be odd integers  with $r\geq 2^{t_1} +1$, $s \geq 2^{t_2}$ (or $s\geq 2^{t_1} +1$, $r \geq 2^{t_2}$) for some positive integer $t_1, t_2$. Then 
 \begin{equation}\label{eq: tcS1 ineq}
n(2^{t_1}+ 2^{t_2})-1 \leq \TC_{S^1,n}(\R H_{r,s})\leq  n(s+r-1).
\end{equation}
\end{proposition}
\begin{proof}
Since both $r$ and $s$ are odd, \Cref{thm: free S1 action} implies that $\R H_{r,s}$ admits a free $S^1$-action.
 Then the right inequality of \eqref{eq: tcS1 ineq} follows from \Cref{prop: eqtcn improved ub}.
 The left inequality of \eqref{eq: tcS1 ineq} can be derived from the analogous calculations to those presented in \Cref{prop: z2eqtcn}
\end{proof}

\begin{proposition}
 Let $r$ and $s$ be odd integers with $r\geq 2^t$ for some positive integer $t$.  
 Then, for $n\geq 3$, we have
 \begin{equation}
  n(2^t+s-1)-s+2  \leq  \TC_{S^1,n}(\R H_{r,s})\leq n(s+r-1). 
 \end{equation}
\end{proposition}
\begin{proof}
The proof is similar to that of \Cref{prop: s1eqtcn}.
\end{proof}

\vspace{0.6cm}

\noindent \textbf{Declaration of competing interest:}
The authors declare that they have no known competing financial interests or personal relationships that
could have appeared to influence the work reported in this paper.

\bibliographystyle{plain} 
\bibliography{references}
\end{document}